\documentclass[11pt,a4paper]{article}

\title{\bf A characterization of the rate of change of $\Phi$-entropy via
an integral form curvature-dimension condition}
\author{Dejun Luo\footnote{Email: luodj@amss.ac.cn. Partly supported by the Key Laboratory of
RCSDS, CAS (2008DP173182), NSFC (11371099) and AMSS (Y129161ZZ1).}
\vspace{3mm}\\
{\footnotesize Institute of Applied Mathematics, Academy of Mathematics and Systems Science,}\\
{\footnotesize Chinese Academy of Sciences, Beijing 100190, China}
}
\date{}

\usepackage{amssymb,amsmath,amsfonts,amsthm,color}

\def\E{\mathbb{E}}

\def\R{\mathbb{R}}

\def\I{\mathcal{I}}
\def\d{\textup{d}}
\def\Ric{\textup{Ric}}
\def\Hess{\textup{Hess}}
\def\Ent{\textup{Ent}}

\def\II{{\rm I\!I}}
\def\<{\langle}
\def\>{\rangle}

\newtheorem{theorem}{Theorem}[section]
\newtheorem{lemma}[theorem]{Lemma}
\newtheorem{corollary}[theorem]{Corollary}
\newtheorem{proposition}[theorem]{Proposition}
\newtheorem{remark}[theorem]{Remark}
\newtheorem{example}[theorem]{Example}

\begin{document}

\maketitle
\makeatletter 
\renewcommand\theequation{\thesection.\arabic{equation}}
\@addtoreset{equation}{section}
\makeatother 

\vspace{-6mm}

\begin{abstract}
Let $M$ be a compact Riemannian manifold without boundary and $V:M\to \R$ a smooth function.
Denote by $P_t$ and $\d\mu=e^V\,\d x$ the semigroup and symmetric measure of the second order
differential operator $L=\Delta+\nabla V\cdot\nabla$. For some suitable convex function
$\Phi:\I\to\R$ defined on an interval $\I$, we consider the $\Phi$-entropy of $P_t f$ (with
respect to $\mu$) for any $f\in C^\infty(M,\I)$. We show that an integral form curvature-dimension
condition is equivalent to an estimate on the rate of change of the $\Phi$-entropy.
We also generalize this result to bounded smooth domains of a complete Riemannian manifold.
\end{abstract}



\section{Introduction and main results}

Let $M$ be an $n$-dimensional compact Riemannian manifold without boundary. The compactness
of the manifold makes it much easier to differentiate under the integral sign and to apply the
integration by parts formula. It also ensures the entropies defined below are always finite,
which allows us to avoid technical difficulties. Denote by $\Delta$ the Laplacian--Beltrami operator and
$\nabla$ the gradient operator, respectively. Consider the diffusion operator
  $$L=\Delta+\nabla V\cdot\nabla,$$
where $V\in C^\infty(M)$. Subtracting a constant from $V$ if necessary, we may assume $\d\mu=
e^{V(x)}\,\d x$ is a probability measure. It is well known that $L$ is symmetric with respect to $\mu$:
  $$\int_M fLg\,\d\mu=\int_M gLf\,\d\mu,\quad \mbox{for all } f,g\in C^2(M).$$
Let $P_t=e^{tL}$ be the heat semigroup associated to $L$. Then for any $f\in C^2(M)$,
  $$\frac{\partial}{\partial t}P_tf=P_tL f=L P_tf.$$

Now we define the ``carr\'e du champ'' operator associated to $L$ which was introduced in
\cite{BakryEmery}: for $f,g\in C^2(M)$,
  $$\Gamma(f,g)=\frac12\big(L(fg)-fLg-g Lf\big)=\nabla f\cdot\nabla g$$
and
  \begin{align*}
  \Gamma_2(f,g)&=\frac12\big[L\Gamma(f,g)-\Gamma(Lf,g)-\Gamma(f,Lg)\big]\\
  &=\<\Hess_f,\Hess_g\>_{HS}+\big(\Ric-\Hess_V\big)(\nabla f,\nabla g),
  \end{align*}
where $\Hess$ and $\Ric$ are respectively the Hessian tensor and the Ricci curvature tensor, and
$\<\, ,\>_{HS}$ denotes the inner product of matrices corresponding to the Hilbert--Schmidt norm.
To simplify notations, we shall write
  $$\Gamma(f)=\Gamma(f,f)\quad \mbox{and}\quad \Gamma_2(f)=\Gamma_2(f,f).$$

We can now present the well known curvature-dimension condition due to Bakry and \'Emery.
The operator $L$ is said to satisfy the curvature-dimension condition $CD(K,m)$ for
some $K\in\R$ and $m>0$ if
  \begin{equation}\label{CDC}
  \Gamma_2(f)\geq K\Gamma(f)+\frac1m (Lf)^2\quad  \mbox{for any } f\in C^\infty(M).
  \end{equation}
It is equivalent to
  $$m\geq n\quad \mbox{and} \quad (m-n)[\Ric-\Hess_V -K]\geq \nabla V\otimes\nabla V.$$
Recently, Baudoin and Garofalo proposed in \cite{Baudoin} a generalized curvature-dimension
condition, allowing us to deal with the sub-elliptic operators (see also Wang \cite{Wang12}).
We mention that a number of semigroup properties which are equivalent to \eqref{CDC}
have been found by F.Y. Wang in \cite[Theorem 1.1]{Wang11}, see also \cite[Proposition 3.3]{Bakry}
or \cite[Theorem 2.3.3]{Wang-book} for the case where $m=\infty$. In the setting of metric
measure spaces, Ambrosio et al. introduced in \cite{Ambrosio} a weak version of the
Bakry--\'Emery curvature-dimension condition which coincides with \eqref{CDC} if the space is smooth.
The second order symmetric covariant tensor field $\Ric-\Hess_V$ is called the Bakry--\'Emery Ricci
tensor, for which the prescribing curvature problem (i.e., finding a metric such that its Bakry--\'Emery Ricci curvature fulfills some prescribed properties) was shown to be solvable in the conformal
class if the initial Bakry--\'Emery Ricci tensor belongs to a negative cone, see [20] for details.

Notice that the condition \eqref{CDC} is a pointwise inequality in the sense that, for any
given $f\in C^\infty(M)$, it holds for all $x\in M$. When $K>0$, it was proved in
\cite[Corollaire 1, p.199]{BakryEmery} that the following integral form condition
  $$\int_M e^f \Gamma_2(f)\,\d\mu\geq K\int_M e^f\Gamma(f)\,\d\mu \quad \mbox{for all } f\in C^\infty(M)$$
implies that the probability measure $\mu$ satisfies the log-Sobolev inequality.
On the other hand, the inverse implication does not hold:
one can find in \cite[Example 5.5.7]{Ane} a measure $\mu$ which fulfills the log-Sobolev inequality
but dissatisfies the above integral inequality. Our purpose is to show that such an integral condition is
in fact equivalent to an estimate on the rate of decay of the relative entropy
of solutions to heat equations corresponding to $L$.

We first introduce some notations. Let $\Phi:\I\to\R$ be a smooth convex function
defined on an interval $\I\subset\R$, such that $\Phi''$ and $-1/\Phi''$ are also convex.
Typical examples are $\Phi(x)=x\log x$ and $\Phi(x)=\frac{x^p-x}{p(p-1)} \, (1<p\leq 2)$
on $\I=\R_+$, and $\Phi(x)=x^2$ on $\R$.

Fix any $f\in C^\infty(M)$ taking values in $\I$. We define the $\Phi$-entropy of $f$ as follows:
  $$\Ent_\Phi(f)=-\int_M \Phi(f)\,\d\mu.$$
The readers can find in \cite{Chafai} a comprehensive study of the $\Phi$-entropy and its
relation with the convexity and functional inequalities. We mention that this framework is not
restricted to the diffusion operators $L$, but it also works in the jump case.
For instance, by exploring the martingale representation approach, a new modified
log-Sobolev inequality for a Poisson space has been obtained earlier in \cite{Wu}, which includes several
known inequalities as special cases. Similar inequalities are established recently in \cite{Wang14}
for a class of stochastic differential equations driven by purely jump L\'evy processes, based on
the $\Phi$-entropy inequality derived in \cite{Wu, Chafai}. In the current paper, however,
we shall concentrate on the diffusion case.

By the integration by parts formula, the rate of change of the $\Phi$-entropy $\Ent_\Phi(P_t f)$
is expressed by
  $$\frac{\d}{\d t}\Ent_\Phi(P_t f)=\int_M \Phi''(P_tf)\Gamma(P_t f)\,\d\mu\geq 0$$
since $\Phi$ is convex. When $L=\Delta$ (i.e. $V\equiv 0$) and $\Phi(x)=x\log x$, an asymptotic
estimate on the rate of change of $\Ent_\Phi(P_t f)$ was given in \cite[Theorem 1.1]{LimLuo}, provided
that the Ricci curvature satisfies $\Ric\geq K\in\R$. This work was motivated by L. Ni's
papers \cite{Ni04a, Ni04b}, where the author derived the formula for the time derivative of Perelman's
$\mathcal W$-entropy along solutions to the linear heat equation. Ni's results were recently extended
by X.-D. Li \cite{LiXD} to the Witten Laplacian operator on complete Riemannian manifolds,
under suitable Bakry--\'Emery curvature-dimension conditions. In this framework, B. Qian \cite{Qian}
obtained similar estimates as those in \cite{LimLuo}
on the general $\Phi$-entropy under the condition \eqref{CDC}.

To state the main result of this note, we need two more notations:
  $$q_\Phi(f)=\int_M\Phi''(f)\Gamma(f)\,\d\mu\quad \mbox{and} \quad
  C_\Phi(f)=\int_M f^2\Phi''(f)\,\d\mu.$$

\begin{theorem}\label{equivalence}
Fix $K\in\R$ and $m>0$. Then for any $f\in C^\infty(M)$ taking values in $\I$,
  \begin{equation}\label{rate-of-change}
  q_\Phi(P_t f)\leq e^{-2Kt}\bigg[\frac1{q_\Phi(f)}
  +\frac{1-e^{-2Kt}}{mKC_\Phi(f)}\bigg]^{-1}  \quad \mbox{for all } t>0
  \end{equation}
if and only if the following integral form curvature-dimension condition holds: for all
$f\in C^\infty(M,\I)$,
  \begin{equation}\label{integral-CDC}
  \int_M \bigg[\frac{2\,\Gamma_2(\Phi'(f))}{\Phi''(f)}
  -\bigg(\frac1{\Phi''}\bigg)''(f)\big(\Phi''(f)\Gamma(f)\big)^2\bigg]\d\mu
  \geq 2Kq_\Phi(f)+\frac{2(q_\Phi(f))^2}{m C_\Phi(f)}.
  \end{equation}
\end{theorem}

We remark that when $K=0$, the right hand side of \eqref{rate-of-change} is understood to be
the limit as $K\to0$, thus
  $$q_\Phi(P_t f)\leq \bigg[\frac1{q_\Phi(f)}+\frac{2t}{mC_\Phi(f)}\bigg]^{-1}
  =\frac{mq_\Phi(f)C_\Phi(f)}{m C_\Phi(f)+ 2t q_\Phi(f)}.$$
At first glance, the inequalities \eqref{rate-of-change} and \eqref{integral-CDC} look a
little complicated, thus we first give some examples and remarks to help understand them.

\begin{example}\label{example-1}
{\rm We consider the following three cases of $\Phi$-entropies.
\begin{itemize}
\item[(i)] Relative entropy: $\Phi(x)=x\log x,\,x\in(0,\infty)$. In this case, we have
  $$q_\Phi (f)=\int_M \frac{\Gamma(f)}f\,\d\mu=\int_M f\,\Gamma(\log f)\,\d\mu,
  \quad C_\Phi(f)=\int_M f\,\d\mu=\mu(f).$$
Noting that $\big(\frac1{\Phi''}\big)''(x)=0$, the inequality \eqref{integral-CDC} becomes
  \begin{equation}\label{case-1.1}
  \int_M f\,\Gamma_2(\log f)\,\d\mu\geq K\int_M f\,\Gamma(\log f)\,\d\mu
  +\frac1{m\mu(f)} \bigg(\int_M f\,\Gamma(\log f)\,\d\mu\bigg)^2,
  \end{equation}
which can be rewritten as
  \begin{equation}\label{case-1.2}
  \int_M e^f\Gamma_2(f)\,\d\mu\geq K\int_M e^f\Gamma(f)\,\d\mu
  +\frac1{m\mu(e^f)} \bigg(\int_M e^f\Gamma(f)\,\d\mu\bigg)^2.
  \end{equation}

\item[(ii)] Variance: $\Phi(x)=x^2,\,x\in\R$. It holds that
  $$q_\Phi (f)=2\int_M \Gamma(f)\,\d\mu,\quad
  C_\Phi(f)=2\int_M f^2\,\d\mu=2\mu(f^2).$$
Again one has $\big(\frac1{\Phi''}\big)''(x)=0$, thus the inequality \eqref{integral-CDC} becomes
  \begin{equation}\label{case-2}
  \int_M \Gamma_2(f)\,\d\mu\geq K\int_M \Gamma(f)\,\d\mu
  +\frac1{m\mu(f^2)} \bigg(\int_M \Gamma(f)\,\d\mu\bigg)^2.
  \end{equation}

\item[(iii)] Interpolation between the above two cases: for some $1<p\leq 2$,
$\Phi(x)=\frac{x^p-x}{p(p-1)},\,x\in(0,\infty)$. Noticing that $\Phi''(x)=x^{p-2}$, we have
  $$q_\Phi (f)=\int_M f^{p-2}\Gamma(f)\,\d\mu=\int_M f^p\Gamma(\log f)\,\d\mu,\quad
  C_\Phi(f)=\int_M f^2f^{p-2}\,\d\mu=\mu(f^p).$$
Moreover, as $\big(\frac1{\Phi''}\big)''(x)=(2-p)(1-p)x^{-p}$, it follows that \eqref{integral-CDC} becomes
  \begin{equation}\label{case-3}
  \begin{split}
  &\hskip13pt\int_M \frac{f^{2-p}}{(p-1)^2}\Gamma_2(f^{p-1})\,\d\mu
  +\frac{(2-p)(p-1)}2\int_M f^p\Gamma(\log f)^2\,\d\mu\\
  &\geq K\int_M f^p\Gamma(\log f)\,\d\mu
  +\frac1{m\mu(f^p)} \bigg(\int_M f^p\Gamma(\log f)\,\d\mu\bigg)^2.
  \end{split}
  \end{equation}

\end{itemize}}
\end{example}

\begin{remark}\label{remark-1} {\rm
In the case that $m=\infty$, the estimate \eqref{rate-of-change} becomes
$q_\Phi(P_t f)\leq e^{-2Kt} q_\Phi(f)$. Integrating this inequality from $0$ to $t$ yields
  $$\int_M \Phi(f)\,\d\mu-\int_M \Phi(P_t f)\,\d\mu\leq \frac{1-e^{-2Kt}}{2K}\int_M \Phi''(f)\Gamma(f)\,\d\mu.$$
We have $P_tf\to \mu(f)$ as $t$ tends to $\infty$ since the manifold is compact. If $K>0$, then letting
$t\to \infty$ leads to the $\Phi$-Sobolev inequality (cf. the proof of
\cite[Proposition 5, p.198]{BakryEmery}):
  \begin{equation}
  \int_M \Phi(f)\,\d\mu-\Phi(\mu(f))\leq \frac1{2K}\int_M \Phi''(f)\Gamma(f)\,\d\mu.
  \end{equation}
Corresponding to the three cases in Example \ref{example-1}, the $\Phi$-Sobolev inequalities
take the following forms:
\begin{itemize}
\item[(i)] Log-Sobolev inequality: for all $f\in C^\infty(M,\R_+)$,
  $$\int_M f\log\frac f{\mu(f)}\,\d\mu\leq \frac1{2K}\int_M \frac{|\nabla f|^2}f\,\d\mu$$
which, by changing $f$ into $f^2$, becomes
  $$\int_M f^2\log\frac{f^2}{\mu(f^2)}\,\d\mu\leq \frac2{K}\int_M |\nabla f|^2\,\d\mu.$$
\item[(ii)] Poincar\'e inequality:
  \begin{equation}\label{PI}
  {\rm Var}_\mu(f):=\int_M f^2\,\d\mu-\mu(f)^2\leq \frac1K\int_M |\nabla f|^2\,\d\mu.
  \end{equation}
\item[(iii)] Interpolation between the above two inequalities:
  $$\mu(f^p)-\mu(f)^p\leq\frac{p(p-1)}{2K}\int_M f^p|\nabla \log f|^2\,\d\mu.$$
\end{itemize}
}
\end{remark}

The next result asserts that the pointwise curvature-dimension
condition \eqref{CDC} implies the integral form condition \eqref{integral-CDC}.

\begin{proposition}\label{prop-1}
The integral form curvature-dimension condition \eqref{integral-CDC} is a consequence of
the pointwise curvature-dimension condition \eqref{CDC}.
\end{proposition}

As mentioned at the beginning of this section, we assume the manifold is compact to
facilitate the applications of differentiation under the integral sign and of integration
by parts formula. If these are justified, then our results also hold on non-compact manifolds.
In the one dimensional special case, we can provide an example which fulfills the integral form
curvature-dimension condition \eqref{integral-CDC} with a better constant than the one in the
pointwsie inequality \eqref{CDC}.

\begin{example}\label{exa-2} {\rm
Let $V\in C^2(\R^1,\R^1)$ be a concave even function satisfying
  $$V''(x)\begin{cases}
  =0, & \mbox{for all } x\in[0,1];\\
  \leq -1, & \mbox{for all } x \geq 2.
  \end{cases}$$
Then there exists $K>0$ such that \eqref{case-2} holds with $m=\infty$.
Since the Ricci curvature vanishes in this case, it is impossible to find a positive constant $K$ such that the pointwise
curvature-dimension condition \eqref{CDC} holds. }
\end{example}

We mention that by Remark \ref{remark-1}(ii), the measure $\d\mu=e^{V(x)}\,\d x$ with $V$
given in the above example satisfies the Poincar\'e inequality \eqref{PI}. On the other hand,
by \cite[Theorem 1.2]{Wang01}, the measure $\mu$ even fulfills the stronger log-Sobolev inequality.
For the moment, however, we are unable to prove that \eqref{case-1.2} holds with some $K>0$ and
$m=\infty$ (which will result in the log-Sobolev inequality for $\mu$).
In \cite[Example 5.5.7]{Ane}, it is shown that if $V(x)=-\alpha(x^4-2x^2)$ and $f(x)=-3\alpha x^2$,
then $\int_{\R}e^f\Gamma_2(f)\,\d\mu<0$ for $\alpha$ big enough. Consequently, there does not exist
$K>0$ such that \eqref{case-1.2} holds with $m=\infty$.

The rest of this paper is organized as follows. We present in Section 2 the proofs of Theorem
\ref{equivalence}, Proposition \ref{prop-1} and Example \ref{exa-2}.
In Subsection 3.1, we extend Theorem \ref{equivalence} to a bounded domain with smooth boundary.
Finally, as an application of the general results, we consider in Subsection 3.2 the measure
having the square of the ground state of a Schr\"odinger operator as density function, and
provide an explicit estimate on the time derivative of the $\Phi$-entropy.

\section{Proofs of the main results}

This section is devoted to proving the results stated in Section 1. The proof of Theorem
\ref{equivalence} consists of two parts: the sufficiency and necessity of \eqref{integral-CDC}.
Both of them are dependent on the following equality.

\begin{lemma}\label{2-lem}
Let $P_t$ be the semigroup generated by $L=\Delta+\nabla V\cdot\nabla$. Then for any
$f\in C^\infty(M,\I)$, we have
  \begin{equation}\label{2-lem.1}
  \bigg(L-\frac{\partial}{\partial t}\bigg)\big[\Phi''(P_t f)\Gamma(P_t f)\big]
  =\frac{2\,\Gamma_2(\Phi'(P_t f))}{\Phi''(P_t f)}
  -\bigg(\frac1{\Phi''}\bigg)''(P_t f)\big[\Phi''(P_t f)\Gamma(P_t f)\big]^2.
  \end{equation}
\end{lemma}

\begin{proof}
It has been proved in \cite[Lemma 2.1]{Qian}. We present it here for the reader's convenience.
By the definition of $\Gamma$ and $\Gamma_2$,
  \begin{align*}
  &\hskip13pt \bigg(L-\frac{\partial}{\partial t}\bigg)\big[\Phi''(P_t f)\Gamma(P_t f)\big]\\
  &=\Gamma(P_t f)L\big[\Phi''(P_t f)\big]+\Phi''(P_t f)L\big[\Gamma(P_t f)\big]
  +2\,\Gamma\big(\Phi''(P_t f),\Gamma(P_t f)\big)\\
  &\hskip13pt -\Gamma(P_t f)\Phi'''(P_tf)LP_t f-2\,\Phi''(P_t f)\Gamma(LP_tf,P_t f)\\
  &=\Phi^{(4)}(P_t f)\Gamma(P_t f)^2+2\,\Phi''(P_tf)\Gamma_2(P_t f)
  +2\,\Phi'''(P_tf)\Gamma\big(P_t f,\Gamma(P_t f)\big).
  \end{align*}
Next, since
  $$\Gamma_2\big(\Phi'(P_t f)\big)=\Phi''(P_t f)^2\Gamma_2(P_t f)
  +\Phi''(P_t f)\Phi'''(P_t f)\Gamma\big(P_t f,\Gamma(P_t f)\big)+\Phi'''(P_t f)^2\Gamma(P_t f)^2,$$
we have
  $$\bigg(L-\frac{\partial}{\partial t}\bigg)\big[\Phi''(P_t f)\Gamma(P_t f)\big]
  =\frac{2\,\Gamma_2(\Phi'(P_t f))}{\Phi''(P_t f)}+\frac{\Gamma(P_t f)^2}{\Phi''(P_t f)}
  \big[\Phi''(P_t f)\Phi^{(4)}(P_t f)-2\,\Phi'''(P_tf)^2\big].$$
Combining this with the identity
  $$-\bigg(\frac1{\Phi''(x)}\bigg)''=\frac1{\Phi''(x)^3}\big[\Phi''(x)\Phi^{(4)}(x)-2\,\Phi'''(x)^2\big]$$
leads to the desired result.
\end{proof}

Now we are ready to prove the main result of this paper.

\begin{proof}[Proof of Theorem \ref{equivalence}]
(i) We first prove the ``if\,'' part. In fact, this has more or less been done in
\cite[Theorem 1.1]{Qian}, except that the condition \eqref{integral-CDC} was replaced by
the pointwsie one \eqref{CDC}. Note that $\int_M L\varphi\,\d\mu=0$ for any $\varphi\in C^\infty(M)$.
Integrating both sides of \eqref{2-lem.1} on $M$ with respect to $\mu$ gives us
  \begin{equation}\label{2.1}
  -\frac{\d}{\d t}\int_M \Phi''(P_t f)\Gamma(P_t f)\,\d\mu
  =\int_M\bigg\{\frac{2\,\Gamma_2(\Phi'(P_t f))}{\Phi''(P_t f)}
  -\bigg(\frac1{\Phi''}\bigg)''(P_t f)\big[\Phi''(P_t f)\Gamma(P_t f)\big]^2\bigg\}\d\mu.
  \end{equation}
Applying \eqref{integral-CDC} with $f$ replaced by $P_t f$ gives us
  $$-\frac{\d}{\d t}q_\Phi(P_t f)\geq 2Kq_\Phi(P_t f)+\frac2{m C_\Phi(P_t f)} \big(q_\Phi(P_t f)\big)^2.$$
Since the function $\I\ni x\mapsto x^2\Phi''(x)$ is convex (cf. \cite[p.330]{Chafai}),
Jensen's inequality leads to
  \begin{equation}\label{C-Phi}
  C_\Phi(P_t f)=\int_M (P_t f)^2\Phi''(P_t f)\,\d\mu\leq \int_M P_t\big(f^2\Phi''(f)\big)\,\d\mu=C_\Phi(f).
  \end{equation}
Therefore,
  $$-\frac{\d}{\d t}q_\Phi(P_t f)\geq 2Kq_\Phi(P_t f)+\frac2{m C_\Phi(f)} \big(q_\Phi(P_t f)\big)^2.$$
Solving this differential inequality yields \eqref{rate-of-change}. \medskip

(ii) For the ``only if\,'' part, we can give two different proofs.

\emph{First proof.} Noticing that the equality holds in \eqref{rate-of-change} at $t=0$, we have
  $$q_\Phi(P_t f)-q_\Phi(f)
  \leq e^{-2Kt}\bigg[\frac1{q_\Phi(f)}+\frac{1-e^{-2Kt}}{mK C_\Phi(f)}\bigg]^{-1}-q_\Phi(f).$$
Dividing both sides by $t$ and letting $t\downarrow 0$, we obtain
  \begin{align}\label{derivative}
  \int_M\frac{\partial}{\partial t}\big(\Phi''(P_tf)\Gamma(P_t f) \big)\Big|_{t=0}\,\d \mu
  &\leq -2Kq_\Phi(f)-\frac2{m C_\Phi(f)}\, (q_\Phi(f))^2.
  \end{align}
Letting $t\downarrow 0$ in \eqref{2.1}, we obtain
  $$-\int_M\frac{\partial}{\partial t}\big(\Phi''(P_tf)\Gamma(P_t f) \big)\Big|_{t=0}\,\d \mu
  =\int_M\bigg\{\frac{2\,\Gamma_2(\Phi'(f))}{\Phi''(f)}
  -\bigg(\frac1{\Phi''}\bigg)''(f)\big[\Phi''(f)\Gamma(f)\big]^2\bigg\}\d\mu.$$
Combining it with \eqref{derivative} yields the inequality \eqref{integral-CDC}.

\emph{Second proof.} Without using the equality \eqref{2.1}, we can give another proof of the
``only if\,'' part by making use of the $\Gamma_2$ calculus and integration by parts formula,
though this proof is much longer than the first one.

We start from \eqref{derivative}. Since
  \begin{align*}
  \frac{\partial}{\partial t}\big(\Phi''(P_tf)\Gamma(P_t f) \big)\Big|_{t=0}&=
  \big[\Phi'''(P_t f)(LP_t f)\Gamma(P_t f)+\Phi''(P_t f)\cdot 2\,\Gamma(LP_t f,P_t f)\big]\big|_{t=0}\cr
  &=\Phi'''(f)(L f)\Gamma( f)+ 2\,\Phi''( f)\Gamma(L f, f)\\
  &=\Phi'''(f)(L f)\Gamma( f)+ 2\,\Gamma(L f, \Phi'( f)),
  \end{align*}
the inequality \eqref{derivative} becomes
  \begin{equation}\label{derivative.1}
  \int_M\big[\Phi'''(f)(L f)\Gamma( f)+ 2\,\Gamma(L f, \Phi'( f))\big]\,\d \mu
  \leq -2Kq_\Phi(f)-\frac2{m C_\Phi(f)}\, (q_\Phi(f))^2.
  \end{equation}

Noticing that $ L\Phi'(f)=\Phi''(f)Lf+\Phi'''(f)\Gamma(f)$, we have
  \begin{equation}\label{proof.1}
  Lf = \frac{L\Phi'(f)}{\Phi''(f)}-\frac{\Phi'''(f)\Gamma(f)}{\Phi''(f)}.
  \end{equation}
As a result,
  \begin{align*}
  \Gamma(Lf,\Phi'(f))&=\Gamma\bigg(\frac{L\Phi'(f)}{\Phi''(f)},\Phi'(f)\bigg)
  -\Gamma\bigg(\frac{\Phi'''(f)\Gamma(f)}{\Phi''(f)},\Phi'(f)\bigg)\\
  &=\frac1{\Phi''(f)}\Gamma\big(L\Phi'(f),\Phi'(f)\big)-\frac{L\Phi'(f)}{(\Phi''(f))^2}\Gamma(\Phi''(f),\Phi'(f))\\
  &\hskip13pt -\frac1{\Phi''(f)}\Gamma\big(\Phi'''(f)\Gamma(f),\Phi'(f)\big)
  +\frac{\Phi'''(f)\Gamma(f)}{(\Phi''(f))^2}\Gamma(\Phi''(f),\Phi'(f)).
  \end{align*}
By \eqref{proof.1},
  \begin{equation}\label{proof.2}
  \Gamma(Lf,\Phi'(f))=\frac1{\Phi''(f)}\Gamma\big(L\Phi'(f),\Phi'(f)\big)
  -\Gamma\big(\Phi'''(f)\Gamma(f),f\big) -\Gamma(\Phi''(f),f)Lf.
  \end{equation}
By the integration by parts formula, we obtain
  $$  -\int_M \Gamma(\Phi''(f),f)Lf\,\d\mu=\int_M\Gamma\big(\Gamma(\Phi''(f),f),f\big)\,\d\mu
  =\int_M\Gamma\big(\Phi'''(f)\Gamma(f),f\big)\,\d\mu.$$
Therefore,
  $$\int_M \Gamma(Lf,\Phi'(f))\,\d\mu=\int_M \frac1{\Phi''(f)}\Gamma\big(L\Phi'(f),\Phi'(f)\big)\,\d\mu.$$
Substituting this identity into \eqref{derivative.1} yields
  \begin{equation}\label{proof.3}
  \int_M \Phi'''(f)(L f)\Gamma( f)\,\d\mu+\int_M \frac{2\,\Gamma\big(L\Phi'(f),\Phi'(f)\big)}{\Phi''(f)}\,\d\mu
  \leq -2Kq_\Phi(f)-\frac2{m C_\Phi(f)} q_\Phi(f)^2.
  \end{equation}

By the definition of the operator $\Gamma_2$, we get
  \begin{equation*}
  2\,\Gamma\big(L\Phi'(f),\Phi'(f)\big)=L\Gamma(\Phi'(f))-2\,\Gamma_2(\Phi'(f)),
  \end{equation*}
Thus
  \begin{equation}\label{proof.4}
  \begin{split}
  &\hskip13pt\int_M \Phi'''(f)(L f)\Gamma( f)\,\d\mu+\int_M \frac{2\,\Gamma\big(L\Phi'(f),\Phi'(f)\big)}{\Phi''(f)}\,\d\mu\cr
  & =-2\int_M\frac{\Gamma_2(\Phi'(f))}{\Phi''(f)}\,\d\mu+\int_M \bigg[ \Phi'''(f)(L f)\Gamma( f)
  +\frac{L\Gamma(\Phi'(f))}{\Phi''(f)}\bigg]\d\mu.
  \end{split}
  \end{equation}
The integration by parts formula leads to
  \begin{equation}\label{proof.5}
  \begin{split}
  \int_M \Phi'''(f)(L f)\Gamma( f)\,\d\mu&=-\int_M\Gamma\big(f,\Phi'''(f)\Gamma( f)\big)\,\d\mu\\
  &=-\int_M \Phi'''(f)\Gamma\big(f,\Gamma( f)\big)\,\d\mu-\int_M \Gamma( f)\Gamma\big(f,\Phi'''(f)\big)\,\d\mu.
  \end{split}
  \end{equation}
Again by the integration by parts formula,
  \begin{align*}
  \int_M \frac{L\Gamma(\Phi'(f))}{\Phi''(f)}\,\d\mu&=\int_M \Gamma(\Phi'(f))L\big[(\Phi''(f))^{-1}\big]\,\d\mu\\
  &=-\int_M \frac{\Gamma(\Phi'(f))}{(\Phi''(f))^2}L\Phi''(f)\,\d\mu
  +2\int_M \frac{\Gamma(\Phi'(f))}{(\Phi''(f))^3}\Gamma(\Phi''(f))\,\d\mu\\
  &=-\int_M \Gamma(f)L\Phi''(f)\,\d\mu +2\int_M \frac{\Gamma(f)}{\Phi''(f)}\Gamma(\Phi''(f))\,\d\mu\\
  &=\int_M \Phi'''(f)\Gamma\big(\Gamma(f),f\big)\,\d\mu +2\int_M \frac{\Gamma(f)}{\Phi''(f)}\Gamma(\Phi''(f))\,\d\mu.
  \end{align*}
Combining this equality with \eqref{proof.5}, we arrive at
  \begin{align*}
  \int_M \bigg[ \Phi'''(f)(L f)\Gamma( f)+\frac{L\Gamma(\Phi'(f))}{\Phi''(f)}\bigg]\d\mu
  &=\int_M \Gamma(f)\bigg[\frac{2\,\Gamma(\Phi''(f))}{\Phi''(f)}-\Gamma\big(f,\Phi'''(f)\big)\bigg]\d\mu\\
  &=\int_M \frac{\Gamma(f)^2}{\Phi''(f)}\big[2(\Phi'''(f))^2-\Phi''(f)\Phi^{(4)}(f)\big]\d\mu\\
  &=\int_M \big[\Gamma(f)\Phi''(f)\big]^2\bigg(\frac1{\Phi''}\bigg)''(f)\,\d\mu.
  \end{align*}
Substituting this identity into \eqref{proof.4} yields
  \begin{align*}
  &\hskip13pt \int_M \Phi'''(f)(L f)\Gamma( f)\,\d\mu+\int_M \frac{2\,\Gamma\big(L\Phi'(f),\Phi'(f)\big)}{\Phi''(f)}\,\d\mu\\
  &=-2\int_M\frac{\Gamma_2(\Phi'(f))}{\Phi''(f)}\,\d\mu
  +\int_M \big[\Gamma(f)\Phi''(f)\big]^2\bigg(\frac1{\Phi''}\bigg)''(f)\,\d\mu.
  \end{align*}
Combining this with \eqref{proof.3} finishes the proof.
\end{proof}

Now we prove Proposition \ref{prop-1}. This has indeed been done in the proof of
\cite[Theorem 1.1]{Qian}, and we present it here for the reader's convenience.

\begin{proof}[Proof of Proposition \ref{prop-1}]
Since $-1/\Phi''$ is convex, we have $(-1/\Phi'')''(x)\geq 0$,
hence it suffices to show that
  \begin{equation}\label{prop-1.1}
  \int_M \frac{\Gamma_2(\Phi'(f))}{\Phi''(f)}\,\d\mu
  \geq Kq_\Phi(f)+\frac{(q_\Phi(f))^2}{m C_\Phi(f)}.
  \end{equation}
By the convexity of $\Phi$ and the curvature-dimension condition \eqref{CDC}, we have
  \begin{equation}\label{prop-1.2}
  \begin{split}
  \int_M \frac{\Gamma_2(\Phi'(f))}{\Phi''(f)}\,\d\mu
  &\geq K\int_M \frac{\Gamma(\Phi'(f))}{\Phi''(f)}\,\d\mu
  +\frac1m \int_M \frac{|L(\Phi'(f))|^2}{\Phi''(f)}\,\d\mu\\
  &=K q_\Phi(f)+\frac1m \int_M \frac{|L(\Phi'(f))|^2}{\Phi''(f)}\,\d\mu.
  \end{split}
  \end{equation}
By the Cauchy inequality,
  \begin{align*}
  \bigg(\int_M f L(\Phi'(f))\,\d\mu\bigg)^2
  &=\bigg(\int_M f\sqrt{\Phi''(f)}\, \frac{L(\Phi'(f))}{\sqrt{\Phi''(f)}}\,\d\mu\bigg)^2\\
  &\leq \int_M f^2\Phi''(f)\,\d\mu\cdot \int_M \frac{|L(\Phi'(f))|^2}{\Phi''(f)}\,\d\mu,
  \end{align*}
therefore
  \begin{equation}\label{prop-1.3}
  \begin{split}
  \int_M \frac{|L(\Phi'(f))|^2}{\Phi''(f)}\,\d\mu
  &\geq \frac1{C_\Phi(f)}\bigg(\int_M f L(\Phi'(f))\,\d\mu\bigg)^2\\
  &= \frac1{C_\Phi(f)}\bigg(\int_M \Gamma(f,\Phi'(f))\,\d\mu\bigg)^2
  =\frac{(q_\Phi(f))^2}{C_\Phi(f)},
  \end{split}
  \end{equation}
where in the first equality we have used the integration by parts formula. Substituting this
inequality into \eqref{prop-1.2} completes the proof.
\end{proof}

Finally we turn to the

\begin{proof}[Proof of Example \ref{exa-2}]
We shall present a proof using basic mathematical analysis.
Fix any $f\in C_c^\infty(\R^1)$. We have
  \begin{equation}\label{proof-exa.1}
  \begin{split}
  \int_{\R^1}\Gamma_2(f)\,\d\mu&=\int_{\R^1} \big[(f'')^2-V''\,(f')^2\big]\,\d\mu\\
  &\geq \int_{\{|x|> 2\}} (f')^2\,\d\mu+e^{V_{\rm min}}\int_{\{|x|\leq 2\}}
  \big[(f'')^2-V''\,(f')^2\big]\,\d x,
  \end{split}
  \end{equation}
where $V_{\rm min}=\min\{V(x):x\in[-2,2]\}$.
We shall show that there is $K_1>0$ such that
  \begin{equation}\label{proof-exa.2}
  I:=\int_{-2}^2 \big[(f'')^2-V''\,(f')^2\big]\,\d x\geq K_1\int_{-2}^2 (f')^2\,\d x.
  \end{equation}

We distinguish two cases: (i) There is $x_0\in[-2,2]$ such that $f'(x_0)=0$. Then
by Cauchy's inequality,
  \begin{align*}
  \int_{-2}^2 (f'(x))^2\,\d x&=\int_{-2}^2 \bigg(\int_{x_0}^x f''(y)\,\d y\bigg)^2\,\d x
  \leq 16 \int_{-2}^2 (f''(y))^2\,\d y \leq 16I
  \end{align*}
since $-V''\geq 0$. Therefore \eqref{proof-exa.2} holds with $K_1=1/16$.

(ii) For any $x\in [-2,2]$ one has $f'(x)\neq 0$. Without loss of generality, we assume
$f'_{\rm min}:=\min\{f'(x):x\in[-2,2]\}>0$ and the minimum is achieved at $x_0\in[-2,2]$,
i.e. $f'(x_0)=f'_{\rm min}$. Again we consider two different cases:
\begin{itemize}
\item[(a)] If $f'(x)\leq 2 f'_{\rm min}$ for all $x\in [-2,2]$, then
  \begin{equation}\label{proof-exa.3}
  \int_{-2}^2 (f'(x))^2\,\d x \leq 16 (f'_{\rm min})^2.
  \end{equation}
On the other hand, we can find $\delta\in(0,1)$ such that for all $2-\delta\leq |x|\leq 2$,
it holds $-V''(x)\geq 1/2$. Note that $\delta$ is independent on the test function $f$. Then
  \begin{align*}
  I\geq \int_{-2}^2 (-V''(x))(f'(x))^2\,\d x&\geq \frac12\int_{-2}^{-2+\delta} (f'(x))^2\,\d x
  +\frac12\int_{2-\delta}^2 (f'(x))^2\,\d x
  \geq \delta (f'_{\rm min})^2.
  \end{align*}
Combining this with \eqref{proof-exa.3}, we obtain the desired inequality \eqref{proof-exa.2}
with $K_1=\delta/16$.

\item[(b)] If there is $x_1\in[-2,2]$ such that $f'(x_1)>2 f'_{\rm min}=2f'(x_0)$, then
by Cauchy's inequality,
  \begin{equation}\label{proof-exa.4}
  (f'(x_0))^2\leq \big[f'(x_1)-f'(x_0)\big]^2 =\bigg(\int_{x_0}^{x_1} f''(x)\,\d x\bigg)^2
  \leq 4\int_{-2}^2 (f''(x))^2\,\d x.
  \end{equation}
Moreover,
  \begin{align*}
  \int_{-2}^2 (f'(x))^2\,\d x&\leq 2\int_{-2}^2 \big[f'(x)-f'(x_0)\big]^2\,\d x +8(f'(x_0))^2\\
  &\leq 2\int_{-2}^2 \bigg[\int_{x_0}^x f''(y)\,\d y\bigg]^2 \d x +32\int_{-2}^2 (f''(x))^2\,\d x\\
  &\leq 32\int_{-2}^2 (f''(y))^2\,\d y+32\int_{-2}^2 (f''(x))^2\,\d x\leq 64I,
  \end{align*}
where in the second and third inequality we have used \eqref{proof-exa.4} and the Cauchy inequality,
respectively. Hence the inequality \eqref{proof-exa.2} holds in this case with $K_1=1/64$.
\end{itemize}

Summarizing the above discussions, we conclude that \eqref{proof-exa.2} holds with $K_1=(\delta/16)
\wedge(1/64)\in (0,1)$. Substituting this result into \eqref{proof-exa.1}, we have
  \begin{align*}
  \int_{\R^1}\Gamma_2(f)\,\d\mu&\geq \int_{\{|x|> 2\}} (f')^2\,\d\mu
  +e^{V_{\rm min}}K_1\int_{\{|x|\leq 2\}}(f')^2\,\d x\\
  &\geq e^{-{\rm osc}(V)}K_1 \int_{\R^1} \Gamma(f)\,\d\mu,
  \end{align*}
where ${\rm osc}(V)=\max\{V(x)-V(y):x,y\in[-2,2]\}$ is the oscillation of $V$ on the interval $[-2,2]$.
\end{proof}

\section{Extension to bounded smooth domains}
\def\A{{\rm \bf A}}

In this section, we assume $M$ is a complete Riemannian manifold and $D\subset M$ a connected
bounded smooth domain. We shall establish in Subsection 3.1 a similar version of Theorem
\ref{equivalence} on the domain $D$. When $M=\R^n$ and $D\subset \R^n$ is convex, we consider in the last subsection
the ground state of the Schr\"odinger operator $-\Delta+U$ on $D$ satisfying the Dirichlet
boundary condition, which can be seen as an application of the general result.

\subsection{An analog of Theorem \ref{equivalence} on a bounded smooth domain}

We first introduce some notations. Let $V\in C^2(\bar D)$ and $L=\Delta+\nabla V\cdot\nabla$.
Denote by $\d\mu =e^V\,\d x$. We write $N$ and $\d\A$ for the inward unit normal vector field and the area
measure of $\partial D$, respectively. Then for any $f,g\in C^2(\bar D)$,
it follows from the integration by parts formula that
  \begin{equation}\label{IBP-boundary}
  \begin{split}
  \int_D gLf\,\d\mu&=\int_D g\Delta f\,\d\mu+\int_D g\nabla V\cdot\nabla f\,\d\mu\\
  &=-\int_D \nabla(ge^V)\cdot\nabla f\,\d x-\int_{\partial D} ge^V Nf\,\d\A+\int_D g\nabla V\cdot\nabla f\,\d\mu\\
  &=-\int_D \nabla g\cdot\nabla f\,\d\mu-\int_{\partial D} ge^V Nf\,\d\A.
  \end{split}
  \end{equation}
Therefore, if $Nf=0$ on $\partial D$, that is, $f$ satisfies the Neumann boundary condition, then it holds
  $$\int_D gLf\,\d\mu=-\int_D \nabla g\cdot\nabla f\,\d\mu=-\int_D \Gamma(f,g)\,\d\mu.$$
where $\Gamma$ is the ``carr\'e du champ'' operator defined in Section 1. Moreover, let $X\in T\partial D$,
the tangent bundle over the boundary $\partial D$, then (here $\nabla$ is also used for the covariant
derivative operator)
  $$0=X(Nf)=\<\nabla_X N,\nabla f\>+\<N,\nabla_X \nabla f\>,$$
therefore
  \begin{equation}\label{Hessian}
  \Hess_f(N,X)=-\<\nabla_X N,\nabla f\>=\II(X,\nabla f),
  \end{equation}
where $\II$ is the second fundamental form of $\partial D$. Next,
  $$N|\nabla f|^2=2\<\nabla_N\nabla f,\nabla f\>=2\,\Hess_f(N,\nabla f),$$
which, together with \eqref{Hessian}, gives us the useful identity
  \begin{equation}\label{2nd-fund-form}
  N|\nabla f|^2=2\,\II(\nabla f,\nabla f) \quad \mbox{for all } f\in C^2(\bar D) \mbox{ with } Nf=0.
  \end{equation}

Let $P_t$ be the semigroup associated to the reflecting diffusion process in $D$ generated by $L$.
Then by \cite[Theorem 3.1.3]{Wang-book}, for $f\in C^\infty(D)$ with $Nf=0$
on $\partial D$, it holds
  \begin{equation}\label{heat-eq}
  \frac{\partial}{\partial t}P_t f=LP_t f=P_t Lf\quad\mbox{and}\quad N(P_t f)|_{\partial D}=0,\quad t\geq0.
  \end{equation}

Let $\Phi:\I\to\R$ be a convex function such that $\Phi''$ and $-/\Phi''$ are convex. Again we define
the $\Phi$-entropy by
  $$\Ent_\Phi(f)=-\int_D \Phi(f)\,\d\mu,$$
where $\d\mu=e^V\,\d x$. Then for $f\in C^1(D)$ with $Nf|_{\partial D}=0$, we have by \eqref{heat-eq}
and \eqref{IBP-boundary} that
  $$\frac{\d}{\d t}\Ent_\Phi(P_t f)=-\int_D \Phi'(P_t f)LP_t f\,\d\mu
  =\int_D \Phi''(P_t f)\Gamma(P_t f)\,\d\mu.$$
We shall partly generalize Theorem \ref{equivalence} to the present setting. To this end,
We still use the notations $q_\Phi(f)=\int_D \Phi''(f)\Gamma(f)\,d\mu$ and
$C_\Phi(f)=\int_D f^2\Phi''(f)\,\d\mu$.

\begin{theorem}\label{3-thm-1}
Let $K\in\R$ and $m>0$. Then for any $f\in C^\infty(D,\I)$ with $Nf|_{\partial D}=0$,
 \begin{equation}\label{3-thm-1.1}
  q_\Phi(P_t f)\leq e^{-2Kt}\bigg[\frac1{q_\Phi(f)}
  +\frac{1-e^{-2Kt}}{mKC_\Phi(f)}\bigg]^{-1}  \quad \mbox{for all } t>0
  \end{equation}
if and only if the following integral form curvature-dimension condition holds: for any
$f\in C^\infty(D,\I)$ with $Nf|_{\partial D}=0$,
  \begin{equation}\label{3-thm-1.2}
  \begin{split}
  &\hskip13pt \int_D \bigg[\frac{2\,\Gamma_2(\Phi'(f))}{\Phi''(f)}
  -\bigg(\frac1{\Phi''}\bigg)''(f)\big(\Phi''(f)\Gamma(f)\big)^2\bigg]\d\mu
  + 2\int_{\partial D} e^V\Phi''(f)\II(\nabla f,\nabla f)\,\d\A \\
  &\hskip43pt \geq 2Kq_\Phi(f)+\frac{2}{m C_\Phi(f)}(q_\Phi(f))^2.
  \end{split}
  \end{equation}
\end{theorem}

Before going into the proofs, we present the following analog of Proposition \ref{prop-1}.

\begin{proposition}\label{3-prop}
If the domain $D$ is convex, then the pointwise curvature-dimension condition \eqref{CDC}
implies the integral form curvature-dimension condition \eqref{3-thm-1.2}.
\end{proposition}

\begin{proof}
Since $D$ is convex, i.e. $\II\geq 0$, it suffices to show that
  $$\int_D \bigg[\frac{2\,\Gamma_2(\Phi'(f))}{\Phi''(f)}
  -\bigg(\frac1{\Phi''}\bigg)''(f)\big(\Phi''(f)\Gamma(f)\big)^2\bigg]\d\mu
  \geq 2Kq_\Phi(f)+\frac{2}{m C_\Phi(f)}(q_\Phi(f))^2.$$
Again by the convexity of $-1/\Phi''$, the above inequality is a consequence of
  $$\int_D \frac{\Gamma_2(\Phi'(f))}{\Phi''(f)}\d\mu
  \geq Kq_\Phi(f)+\frac{1}{m C_\Phi(f)}(q_\Phi(f))^2.$$
The rest of the proof is the same as Proposition \ref{prop-1}, thus we omit it here.
\end{proof}

Now we present the proof of Theorem \ref{3-thm-1} which is similar to those of Theorem \ref{equivalence}.

\begin{proof}[Proof of Theorem \ref{3-thm-1}] (i) Sufficiency. We still have
  \begin{equation}\label{derivative-eq}
  \bigg(L-\frac{\partial}{\partial t}\bigg)\big[\Phi''(P_t f)\Gamma(P_t f)\big]
  =\frac{2\,\Gamma_2(\Phi'(P_t f))}{\Phi''(P_t f)}-\bigg(\frac1{\Phi''}\bigg)''(P_t f)\big[\Phi''(P_t f)\Gamma(P_t f)\big]^2.
  \end{equation}
As in the proof of the ``if\,'' part in Theorem \ref{equivalence}, we shall integrate both
sides with respect to $\mu$ on $D$, but the difference is that the term involving $L$ does
not vanish here. Indeed, by the integration by parts formula \eqref{IBP-boundary},
  \begin{align*}
  \int_D L\big[\Phi''(P_t f)\Gamma(P_t f)\big]\,\d\mu
  &=-\int_{\partial D}e^V N\big[\Phi''(P_t f)\Gamma(P_t f)\big]\,\d\A\\
  &=-\int_{\partial D}e^V \big[\Phi'''(P_t f)\Gamma(P_t f) NP_t f+\Phi''(P_t f)N\Gamma(P_t f)\big]\,\d\A\\
  &=-2\int_{\partial D}e^V \Phi''(P_t f)\II(\nabla P_t f,\nabla P_t f)\,\d\A,
  \end{align*}
where in the last equality we have used the fact that $NP_t f=0$ on $\partial D$ and
\eqref{2nd-fund-form}. Therefore, integrating both sides of \eqref{derivative-eq}, we obtain
  \begin{equation}\label{3-thm-1.3}
  \begin{split}
  -\frac{\d}{\d t}\int_D \Phi''(P_t f)\Gamma(P_t f)\,d\mu
  &=\int_{D}\bigg\{\frac{2\,\Gamma_2(\Phi'(P_t f))}{\Phi''(P_t f)}
  -\bigg(\frac1{\Phi''}\bigg)''(P_t f)\big[\Phi''(P_t f)\Gamma(P_t f)\big]^2\bigg\}\d\mu\\
  &\hskip 13pt +2\int_{\partial D}e^V \Phi''(P_t f)\II(\nabla P_t f,\nabla P_t f)\,\d\A.
  \end{split}
  \end{equation}
Applying the integral form curvature-dimension condition \eqref{3-thm-1.2} with $P_t f$ in
place of $f$ gives us
  \begin{align*}
  -\frac{\d}{\d t}q_\Phi(P_t f)
  &\geq 2Kq_\Phi(P_t f)+\frac2{m C_\Phi(P_t f)}(q_\Phi(P_t f))^2\\
  &\geq 2Kq_\Phi(P_t f)+\frac2{m C_\Phi(f)}(q_\Phi(P_t f))^2,
  \end{align*}
where in the second inequality we have used \eqref{C-Phi}. Solving this inequality leads to
the estimate \eqref{3-thm-1.1} on the rate of change of $\Phi$-entropy.

(ii) Necessity. Noticing that the equality holds in \eqref{3-thm-1.1} at $t=0$, we have
  $$q_\Phi(P_t f)-q_\Phi(f)
  \leq e^{-2Kt}\bigg[\frac1{q_\Phi(f)}+\frac{1-e^{-2Kt}}{mK C_\Phi(f)}\bigg]^{-1}-q_\Phi(f).$$
Dividing both sides by $t$ and letting $t\downarrow 0$, we obtain
  \begin{align}\label{3-thm-1.4}
  \int_M\frac{\partial}{\partial t}\big(\Phi''(P_tf)\Gamma(P_t f) \big)\Big|_{t=0}\,\d \mu
  &\leq -2Kq_\Phi(f)-\frac2{m C_\Phi(f)}\, (q_\Phi(f))^2.
  \end{align}
Letting $t\downarrow 0$ in \eqref{3-thm-1.3} gives us
  \begin{align*}
  -\int_M\frac{\partial}{\partial t}\big(\Phi''(P_tf)\Gamma(P_t f) \big)\Big|_{t=0}\,\d \mu
  &=\int_{D}\bigg\{\frac{2\,\Gamma_2(\Phi'(f))}{\Phi''(f)}
  -\bigg(\frac1{\Phi''}\bigg)''(f)\big[\Phi''(f)\Gamma(f)\big]^2\bigg\}\d\mu\\
  &\hskip 13pt +2\int_{\partial D}e^V \Phi''(f)\II(\nabla f,\nabla f)\,\d\A.
  \end{align*}
Substituting the above equality into \eqref{3-thm-1.4} completes the proof.
\end{proof}

\begin{remark}
As in the proof of the necessity part of Theorem \ref{equivalence}, we can also give another proof
without using equality \eqref{3-thm-1.3}.
\end{remark}

\subsection{The ground state of $-\Delta+U$ on a bounded convex domain}

This part is motivated by \cite[Subsection 2.4]{Gong}, where a log-Sobolev inequality
was established for a measure whose density is given by the ground state of a Schr\"odinger
operator on a convex domain.
From now on, we assume $D\subset \R^n$ is a bounded convex domain and $U$ a smooth
potential function on $\bar D$. Consider the Schr\"odinger operator $-\Delta+U$ on $D$
with Dirichlet boundary condition, which has an increasing sequence of eigenvalues
$\lambda_0<\lambda_1\leq\lambda_2\leq\ldots$, with the associated eigenfunctions
$\{\phi_i\}_{i\geq0}$ which vanish on the boundary $\partial D$. The eigenfunction
$\phi_0>0$ and eigenvalue $\lambda_0$ are also called the ground state and ground state
energy, respectively. In the recent paper \cite{Andrews}, Andrews and Clutterbuck
proved the fundamental gap conjecture which states that if $U$ is convex, then the
spectral gap $\lambda_1-\lambda_0\geq \frac{3\pi^2}{\textup{diam}(D)^2}$, where $\textup{diam}(D)$
is the diameter of the domain $D$ (cf. \cite{Gong} for a probabilistic approach).

Now let $V=\log \phi_0^2=2\log\phi_0$. Although $V$ explodes on the boundary $\partial D$,
the function $e^V=\phi_0^2$ is smooth on the closure $\bar D$, thus we can consider the
measure $\d\mu=\phi_0^2\,\d x$ which will be assumed to be a probability on $D$. It is
easy to see that $\mu$ is a symmetric measure of the diffusion operator
  $$L=\Delta+\nabla(\log\phi_0^2)\cdot\nabla=\Delta+2\nabla\log\phi_0\cdot \nabla,$$
and
  \begin{equation}\label{3-IBP}
  \int_D fLg\,\d\mu=-\int_D \Gamma(f,g)\,\d\mu,\quad  \mbox{for any } f,g\in C^2(\bar D).
  \end{equation}
Compared to \eqref{IBP-boundary}, the integral involving the boundary vanishes since $\phi_0|_{\partial D}\equiv 0$.
Let $P_t$ be the semigroup generated by $L$ which can be constructed as follows. Consider
the It\^o SDE
  \begin{equation}\label{3-SDE}
  \d X_t=\sqrt 2\, \d B_t+2\nabla\log\phi_0(X_t)\,\d t,\quad X_0=x\in D,
  \end{equation}
where $B_t$ is a standard Brownian motion on $\R^n$. It follows from the properties of the ground
state $\phi_0$ that, starting from any point $x$ in the interior of $D$, the process $X_t$ will
not hit the boundary $\partial D$ (see \cite[Lemma 2.8]{Gong} for a proof). Therefore, unlike
\cite[(3.0.1)]{Wang-book}, we do not need to add an reflection term to the right hand
side of \eqref{3-SDE}; moreover, we have
  $$P_t f(x)=\E_x f(X_t),\quad \mbox{for all } x\in D, f\in C(\bar D).$$
Taking into consideration this fact and the integration by parts formula \eqref{3-IBP},
the calculations below are more like those in the case of a manifold without boundary,
cf. Sections 1 and 2.

As before, we take a smooth convex function $\Phi:\I\to\R$ such that $\Phi''$ and $-1/\Phi''$ are
also convex, and consider the $\Phi$-entropy
  $$\Ent_\Phi(f)=-\int_D \Phi(f)\,\d\mu=-\int_D \Phi(f)\phi_0^2\,\d x,\quad f\in C^\infty(\bar D,\I).$$
We shall give an estimate on the rate of change of the $\Phi$-entropy $\Ent_\Phi(P_t f)$, based on
Andrews and Clutterbuck's estimate on the modulus of concavity of $\log\phi_0$ (cf. \cite[Theorem 1.5]{Andrews} or \cite[Theorem 2.11]{Gong}).
Recall that a function $\tilde U\in C^1([0,\textup{diam}(D)/2])$ is called a modulus
of concavity of $U\in C^1(\bar D)$ if for any $x,y\in \bar D, x\neq y$, one has
  $$\Big\<\nabla U(x)-\nabla U(y),\frac{x-y}{|x-y|}\Big>\leq 2\tilde U'\bigg(\frac{|x-y|}2\bigg).$$
If `$\leq$' is replaced by `$\geq$', then $\tilde U$ is called a  modulus of convexity of $U$.

\begin{theorem}\label{3-thm-2}
Assume that the potential $U\in C^1(\bar D)$ admits a modulus of convexity
$\tilde U\in C^1([-\textup{diam}(D)/2,\textup{diam}(D)/2])$ which is an even function.
Denote by $\tilde\lambda_0$ the first Dirichlet eigenvalue of the one dimensional
Schr\"odinger operator $-\frac{\d^2}{\d t^2}+\tilde U$ on $[-\textup{diam}(D)/2,\textup{diam}(D)/2]$.
Then for any $f\in C^\infty(\bar D,\I)$,
  \begin{equation}\label{3-thm-2.1}
  q_\Phi(P_t f)\leq e^{4t(\tilde V(0)-\tilde\lambda_0)}q_\Phi(f),
  \end{equation}
where $q_\Phi(f)=\int_D \Phi''(f)\Gamma(f)\,\d\mu$.
\end{theorem}

\begin{proof}
The second order differential operator is now given by $L=\Delta+2\nabla\log\phi_0\cdot \nabla$,
thus we have
  $$\Gamma(f)=|\nabla f|^2\quad \mbox{and}\quad \Gamma_2(f)=\|\Hess_f\|_{HS}^2
  -2\,\Hess_{\log\phi_0}(\nabla f,\nabla f).$$
Let $\tilde\phi_0$ be the eigenfunction of $-\frac{\d^2}{\d t^2}+\tilde U$
corresponding to $\tilde\lambda_0$ which is strictly positive
on the open interval $(-\textup{diam}(D)/2,\textup{diam}(D)/2)$. Since $\tilde U$ is even,
it is easy to show that $\tilde\phi_0$ is also even, hence $\tilde\phi'_0(0) = 0$.
By \cite[Theorem 1.5]{Andrews}, we know that $\log\tilde\phi_0$ is a modulus of
concavity of $\log\phi_0$, that is, for all $x,y\in D$ with $x\neq y$,
  $$\Big\<\nabla \log\phi_0(x)-\nabla \log\phi_0(y),\frac{x-y}{|x-y|}\Big>
  \leq 2 (\log\tilde\phi_0)'\bigg(\frac{|x-y|}2\bigg).$$
From this it is easy to show that
  $$\Hess_{\log\phi_0}(x)\leq (\log\tilde\phi_0)''(0)=\frac{\tilde\phi_0''(0)}{\tilde\phi_0(0)},
  \quad \mbox{for all } x\in D.$$
Using the eigen-equation $-\tilde\phi_0''+\tilde U \tilde\phi_0=\tilde\lambda_0 \tilde\phi_0$, we have
$\tilde\phi_0''(0)=(\tilde U(0)-\tilde\lambda_0)\tilde\phi_0(0)$. Therefore
  $$\Hess_{\log\phi_0}(x)\leq \tilde U(0)-\tilde\lambda_0.$$
By the expressions of $\Gamma$ and $\Gamma_2$, we obtain
  $$\Gamma_2(f)\geq  2\big(\tilde\lambda_0-\tilde U(0)\big)\Gamma(f),
  \quad \mbox{for all } f\in C^\infty(\bar D).$$
That is, the curvature-dimension condition \eqref{CDC} holds with $K=2\big(\tilde\lambda_0-\tilde U(0)\big)$
and $m=\infty$. Therefore, the same argument as Proposition \ref{prop-1} implies that \eqref{3-thm-1.2}
holds with the same $K$ and $m$, and the term involving the integral on the boundary $\partial D$ vanishes.
By Theorem \ref{3-thm-1}, we obtain the desired estimate.
\end{proof}

As mentioned in Remark \ref{remark-1}, if $\tilde\lambda_0>\tilde U(0)$, then by integrating
\eqref{3-thm-2.1} from $t=0$ to $\infty$, we obtain the $\Phi$-Sobolev inequality for the measure
$\d\mu=\phi_0^2\,\d x$ (see \cite[Theorem 2.10]{Gong} for a similar result with $\d\mu=\phi_0\,\d x$).
In the special case of a convex potential $U$, the constant in the estimate
\eqref{3-thm-2.1} on the rate of change of entropy is explicit.

\begin{corollary}
Assume that the potential $U\in C^1(\bar D)$ is convex. Then we have $q_\Phi(P_t f)\leq
e^{-4t\pi^2/\textup{diam}(D)^2}q_\Phi(f)$ for all $f\in C^\infty(\bar D,\I)$.
\end{corollary}

\begin{proof}
Since $U$ is convex, its modulus of convexity is simply given by $\tilde U\equiv 0$. The first Dirichlet
eigenvalue of the differential operator $-\frac{\d^2}{\d t^2}$ on the interval
$[-\textup{diam}(D)/2,\textup{diam}(D)/2]$ is $\tilde\lambda_0=\pi^2/\textup{diam}(D)^2$. Then the estimate
follows from Theorem \ref{3-thm-2}.
\end{proof}

\noindent\textbf{Acknowledgements.} The author is very grateful to Professors Liming Wu for helpful discussions,
and to Bin Qian for his suggestion of proving the necessity part of Theorem \ref{equivalence} by using
the identity \eqref{2-lem.1}, which simplifies the argument.

\end{document}